\documentclass[12pt,a4paper]{amsart}

\usepackage{amsmath, amsthm, amscd, amsfonts, amssymb, graphicx, color}
\usepackage{color}
\usepackage{url}
\usepackage{tikz-cd}
\usepackage{xcolor}

\usepackage[T1]{fontenc}
\newcommand{\camino}[2][q]{P_{#2}^{#1}}

\newcommand{\B}[2][q]{\mathcal{B}(#2,#1)}

\newtheorem{theorem}{Theorem}[section]
\newtheorem{lemma}{Lemma}[section]

\newtheorem{proposition}{Proposition}
\newtheorem{corollary}{Corollary}
\newtheorem{conjecture}{Conjecture}
\newtheorem{definition}{Definition}

\begin{document}

\title[]{On the spectral radius of block graphs having all their blocks of the same size}

\author[C.M. Conde]{Cristian M. Conde${}^{1,3}$}
\author[E. Dratman]{Ezequiel Dratman${}^{1,2}$}
\author[L.N. Grippo]{Luciano N. Grippo${}^{1,2}$}

\address{${}^{1}$Instituto de Ciencias\\Universidad Nacional de General Sarmiento}
\address{${}^{2}$Consejo Nacional de Investigaciones Cient\'ificas y Tecnicas, Argentina}
\address{${}^{3}$Instituto Argentina de Matem\'atica "Alberto Calder\'on" -  Consejo Nacional de Investigaciones Cient\'ificas y Tecnicas, Argentina}

\email{cconde@campus.ungs.edu.ar}
\email{edratman@campus.ungs.edu.ar}
\email{lgrippo@campus.ungs.edu.ar}


\date{}

\maketitle

\begin{abstract}
Let $\B{n}$ be the class of block graphs on $n$ vertices having all their blocks of the same size. We prove that if $G\in\B{n}$ has at most three pairwise adjacent cut vertices then the minimum spectral radius $\rho(G)$ is attained at a unique graph. In addition, we present a lower bound for $\rho(G)$ when $G\in\B{n}$.
\end{abstract}

\section{Introduction}

The problem of finding those graphs that maximize or minimize the spectral radius of a connected graph on $n$ vertices, within a given graph class $\mathcal H$, have attracted the attention of many researchers. Usually, this kind of problems are solved by means of graphs transformations preserving the number of vertices, so that the resulting graph also belongs to $\mathcal H$, and having a monotone behavior respect to the spectral radius.  We refer to the reader to~\cite{Dragan2015} for more details about this and other techniques. In~\cite{LP1973}, Lov\' asz and Pelik\'an proved that the unique graph with maximum spectral radius among the trees on $n$ vertices is the star $K_{1,n-1}$ and the unique graph with minimum spectral radius is the path $P_n$. As far as we know, this article is the first one within this research line. Since adding edges to a graph increases the spectral radius (see Corollary~\ref{cor: adding edges}), if $\mathcal H$ contains complete graphs and paths, then $K_n$ maximizes and $P_n$ minimizes $\rho(G)$ among graphs in $\mathcal H$, meaning that this two graphs have the minimum and maximum spectral radius among graphs on $n$ vertices when $\mathcal H$ is the class of all connected graphs. Consequently, several authors have considered the problem when $\mathcal H$ is a graph class not containing either paths or complete graphs and defined by certain restriction of classical graph parameters. Graphs with a given independence number~\cite{LuLin2015, XS2013}, graphs with a given clique number~\cite{SH2008} and graphs with given connectivity and edge-connectivity~\cite{LSCC2009}. It is worth mentioning that the foundation stone that gives place to many late articles in connection with this problem is that of Brualdi and Solheid~\cite{BS-1986}. 
\subsection*{About our statement in connection with Lov\'asz an Pelik\'an result}
For concepts and definitions used in this section we referred the reader to Section~\ref{sec: preliminaries}. 

In this article we consider the class $\B{n}$ of block graphs on $n$ vertices having all their blocks on $q+1$ vertices, for every $q\ge 2$. For results related to the adjacency matrix of block graphs we refer to the reader to~\cite{brs-2014}. Trees are block graphs with all their blocks on two vertices. In connection with the spectral radius on trees it was obtained the following result.

\begin{theorem}\cite{LP1973}\label{thm: maximum-trees}
	If $T$ is a tree on $n$ vertices, then $2\cos\left(\frac{\pi}{n+1}\right)=\rho(P_n)\le\rho(T)\le\rho(K_{1,n-1})=\sqrt{n-1}$. 
\end{theorem}
In an attempt to generalize Theorem~\ref{thm: maximum-trees}, we find the unique graph in $G\in\B{n}$ that reaches the minimum spectral $\rho(G)$ in the case in which $G$ has at most three pairwise adjacent cut vertices. Besides, we present a lower bound for $\rho(G)$.
\begin{theorem}\label{thm:extremal radius}
	If $G\in\B{n}$, then $\rho(G)\le\rho(S(n,q))$ and  $S(n,q)$ is the unique graph that maximizes the spectral radius. In addition, if $G$ has at most three pairwise adjacent cut vertices then $\rho\left(\camino{b}\right)\le\rho(G)$ and $\camino{b}$ is the unique graph that minimizes the spectral radius in the class $\B{n}$, where $b=\frac{n-1}{q}$. 
\end{theorem}
We have strongly evidence obtained by the aid of Sage software that the hypothesis of having at most three pairwise adjacent cut vertices,  in connection with the minimum of the spectral radius, can be dropped.  
\subsection*{Organization of the article}

This article is organized as follows. In Section~\ref{sec: preliminaries} we present some definitions and preliminary results. In Section~\ref{sec: graphs transformations} are presented two graph transformations having a monotone behavior respect to the spectral radius. Section~\ref{sec: main result} is devoted to put all previous result together in order to prove our main result.  In Section~\ref{sec: bounds} a lower bound for the spectral radius is presented. Finally, Section~\ref{sec: discussion}  contains a short summary of our work and two conjectures are posted.

\section{Preliminaries}\label{sec: preliminaries}

\subsection{Definitions}

All graphs, mentioned in this article, are finite, have no loops and multiple edges. Let $G$ be a graph. We use $V(G)$ and $E(G)$ to denote the set of vertices and the set of edges of $G$, respectively. A graph on one vertex is called \emph{trivial graph}. Let $v$ be a vertex of $G$, $N_G(v)$ (resp. $N_G[v]$) stands for the neighborhood of $v$ (resp. $N_G(v)\cup\{v\}$), if the context is clear the subscript $G$ is omitted. We use $d_G(v)$ to denote the degree of $v$ in $G$, or $d(v)$ provided the context is clear. By $\overline G$ we denote the complement graph of $G$. Given a set $F$ of edges of $G$ (resp. of $\overline G$), we denote by $G-F$ (resp. $G+F$) the graph obtained from $G$ by removing (resp. adding) all the edges in $F$. If $F=\{e\}$, we use $G-e$  (resp. $G+e$) for short. Let $X\subseteq V(G)$, we use $G[X]$ to denote the graph induced by $X$. By $G-X$ we denote the graph $G[V(G)\setminus X]$. If $X=\{v\}$, we use $G-v$ for short. Let $G$ and $H$ \textcolor{red}{be} two graphs, we use $G+H$ to denote the disjoint union between $G$ and $H$, and $G^+$ stands for the graph obtaining by adding an isolated vertex to $G$. We denote by $P_n$ and $K_n$ to the path and the complete graph on $n$ vertices.

We denote by $A(G)$ the adjacency matrix of $G$, and $\rho(G)$ stands for the spectral radius of $A(G)$, we refer to $\rho(G)$ as the spectral radius of $G$. If $x$ is the principal eigenvector of $A(G)$, which is indexed by $V(G)$, we use $x_u$ to denote the coordinate of $x$ corresponding to the vertex $u$. We use $P_G(x)$ to denote the characteristic polynomial of $A(G)$; i.e., $P_G(x)=\det(xI_n-A(G))$. It is easy to prove that $P_{K_n}(x)=(x-n+1)(x+1)^{n-1}$.

A vertex $v$ of a graph $G$ is a \emph{cut vertex} if $G-v$ has a number of connected components greater than the number of connected components of $G$. Let $H$ be a graph. A \emph{block} of $H$, also known as \emph{$2$-connected component}, is a maximal connected subgraph of $H$ having no cut vertex. A \emph{block graph} is a connected graph whose blocks are complete graphs. We use $\B{n}$ to denote the family of block graphs on $n$ vertices whose blocks have $q+1$ vertices.  Notice that if $B\in\B{n}$ and $b$ is its number of blocks then $b=\frac{n-1}{q}$. Let $G$ be a block graph, a \emph{leaf block} is a block of $G$ such that contains exactly one cut vertex of $G$. We use $S(n,q)$ to denote the block graph in $\B{n}$ having $b$ blocks with only one cut vertex. By $\camino{b}$ we denote the block graph in $\B{bq+1}$ with at most two leaf blocks when $n-1>q$ and no cut vertices when $n-1=q$, called $(q,b)$-path-block. 
\subsection{Preliminaries results}
This subsection is split into two parts. In the first one we present the results needed to	deal with the minimum spectral radius in $\B{n}$, and in the second one we briefly describe the previous result in connection with the maximum spectral radius in this class.
\subsubsection*{Tools for the minimum}

We will introduce a partial order on the class of graphs. We will use it to deal with the graph transformations used to prove our main result. This technique was pioneered by Lov\'asz and Pelik\'an~\cite{LP1973}.

\begin{definition}\label{rmk: polynomial comparison}
	Let $G$ and $H$ be two graphs. We denote by $G\prec H$, if $P_H(x)>P_G(x)$ for all $x\ge \rho(G)$. \end{definition}
	It is immediate that if $G\prec H$ then $\rho(H)<\rho(G)$. The spectrum radius is nondecreasing respect to the subgraph partial order.

We repeatedly use the following Lemma to deal with the subgraph partial order previously defined.
\begin{lemma}\label{lem: subgraph spectral radius}
	If $H$ is a proper subgraph of $G$ then $\rho(H)<\rho(G)$.
\end{lemma}

The reader is referred to~\cite{Bapat} for a proof of the above lemma. In particular, adding edges to a graph increases the spectral radius.

\begin{corollary}\label{cor: adding edges}
	If $G$ is a graph such that $uv\notin E(G)$, then $\rho(G)<\rho(G+uv)$. 
\end{corollary}

The following technical lemma is a useful tool to develop graph transformations.

\begin{lemma}\cite{LiFe1979}\label{lem: subgraph poly}
	If $H$ is a spanning subgraph of the graph $G$ then $P_G(x)\le P_H(x)$ for all $x\ge\rho(G)$. In addition, if $G$ is connected then $G\prec H$.
\end{lemma}

Let $G$ and $H$ be two graphs.  If $g\in V(G)$ and $h\in V(H)$,  the \emph{coalescence} between $G$ and $H$ at $g$ and $h$, denoted $G\cdot_{g}^h H$, is the graph obtained from $G$ and $H$, by identifying vertices $g$ and $h$ (see Fig.~\ref{fig: coalescence}). We use $G\cdot H$ for short. Notice that any block graph can be constructed by recursively using  the coalescence operation between a block graph and a complete graph.

\begin{figure}
	\centering
	\includegraphics[scale=0.5]{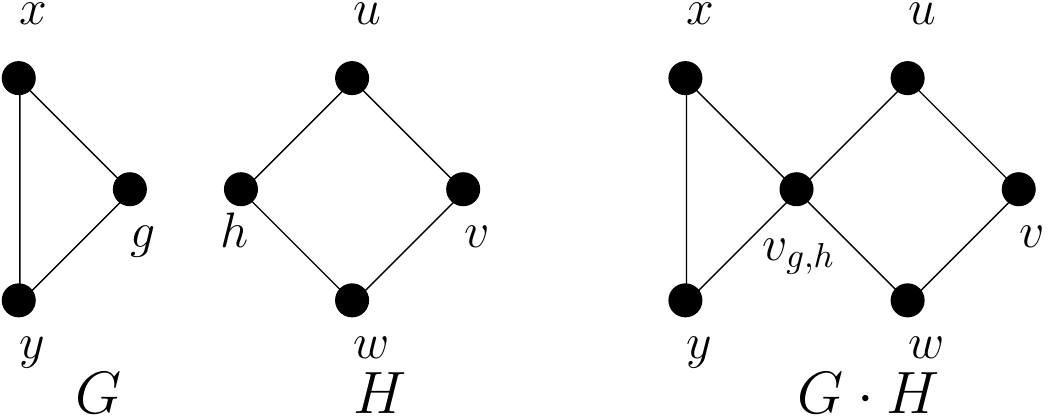}   
	\caption{The coalescence of graphs $G$ and $H$ at vertices $g$ and $h$.}
	\label{fig: coalescence}
\end{figure}

In the 70s Schwenk published an article containing useful formulas for the characteristic polynomial of a graph~\cite{Schwenk1974}. The part corresponding to minimizing the spectral radius of the main result of this research is based on the following Schwenk's formula, linking the characteristic polynomial of two graphs and the coalescence between them.

\begin{lemma}\cite{Schwenk1974}\label{lem: Schwenk}
	Let $G$ and $H$ be two graphs. If $g\in V(G)$, $h\in V(H)$, and $F=G\cdot H$, then
	\[P_F(x)=P_G(x)P_{H-h}(x)+P_{G-g}(x)P_H(x)-xP_{G-g}(x)P_{H-h}(x).\]  
\end{lemma}
More details on Lemmas~\ref{lem: subgraph poly} and~\ref{lem: Schwenk} can be found in~\cite{CRS-1997}.

The following two technical lemmas will play an important role in order to prove the main result of this article.

\begin{lemma}\label{lem: first technical lemma}
	Let $H$ be a graph, let $v,w\in V(H)$ such that $H-w\prec H-v$ and let $G$ be a connected graph. If $G_1$ and $G_2$ are the graphs obtained by means of the coalescence between $G$ and $H$ at $u\in V(G)$, and $v$ or $w$ respectively, then $G_1\prec G_2$. 
\end{lemma}	

\begin{proof}
	By Lemma~\ref{lem: Schwenk}, the characteristic polynomial of $G_1$ and $G_2$ are
	\[P_{G_1}(x)=P_{G-u}(x)P_{H}(x)+(P_G(x)-xP_{G-u}(x))P_{H-v}(x)\]
	and
	\[P_{G_2}(x)=P_{G-u}(x)P_{H}(x)+(P_G(x)-xP_{G-u}(x))P_{H-w}(x)\]
	respectively, and thus
	\begin{equation}\label{eq: primer lema tecnico}
		P_{G_2}(x)-P_{G_1}(x)=(P_G(x)-xP_{G-u}(x))(P_{H-w}(x)-P_{H-v}(x)).
	\end{equation}
	By Lemmas~\ref{lem: subgraph spectral radius} and~\ref{lem: subgraph poly}, $G\prec (G-u)^+$. Therefore, since $H-w\prec H-v$, by Lemma~\ref{lem: subgraph spectral radius}  and~\eqref{eq: primer lema tecnico} we have $G_1\prec G_2$. 	
\end{proof}
%
%
\begin{lemma}\label{lem: second technical lemma}
	Let $H_1,\, H_2$ be two graphs such that either $H_1=H_2$ or $H_1\prec H_2$, let $v_i\in V(H_i)$ for each $i=1,2$ such that $H_2-v_2\prec H_1-v_1$, and let $G$ be a connected graph. If $G_i$ is the graph obtained by means of the coalescence between $G$ and $H_i$ at $v\in V(G)$ and $v_i$ for each $i=1,2$, then $G_1\prec G_2$. 
\end{lemma}

\begin{proof}
	By applying Lemma~\ref{lem: Schwenk} as in Lemma~\ref{lem: first technical lemma} we obtain
	\begin{equation}\label{eq: ecuacion segundo lema tecnico}
	\begin{split}
		P_{G_2}(x)-P_{G_1}(x)&=(P_G(x)-xP_{G-v}(x))(P_{H_2-{v_2}}(x)-P_{H_1-{v_1}}(x))\\
		&+(P_{H_2}(x)-P_{H_1}(x))P_{G-v}(x).
	\end{split}
	\end{equation}
	By Lemmas~\ref{lem: subgraph spectral radius} and~\ref{lem: subgraph poly}, $G\prec (G-v)^+$. Therefore, since either $H_1=H_2$ or $H_1\prec H_2$ and $H_2-v_2\prec H_1-v_1$, by~\eqref{eq: ecuacion segundo lema tecnico} and Lemma~\ref{lem: subgraph spectral radius} we conclude that $G_2\prec G_1$.
\end{proof}

\subsubsection*{Tools for finding the maximum}

%
%
The below theorem was proved in the context of studying the spectral radius of the  adjacency matrix of the $q+1$-uniform hypertrees. This matrix agrees with that of a block graph having all their blocks of size $q+1$. 
\begin{theorem}~\cite[Theorem 4.1]{Linetal2017}\label{thm:extremal radius maximum}
	If $G\in\B{n}$, then $\rho(G)\le\rho(S(n,q))=\frac{q-1+\sqrt{(q-1)^2+4(n-1)}}{2}$. Besides, $S(n,q)$ is the unique graph maximizing $\rho(G)$.
\end{theorem}
The reader should notice that although the minimum spectral radius of a $q+1$-uniform hypertrees was also considered in the literature (see e.g. \cite[Corollary 19]{Zhouetal-2014}) but under the notion of spectral radius of the hypermatrix, also called tensor, associated to a uniform hypertree that does not have an immediate relationship with the notion of the spectral radius of the adjacency matrix of a block graph.

\section{Graph transformations}\label{sec: graphs transformations}
\begin{figure}
	\centering
	\includegraphics[scale=0.5]{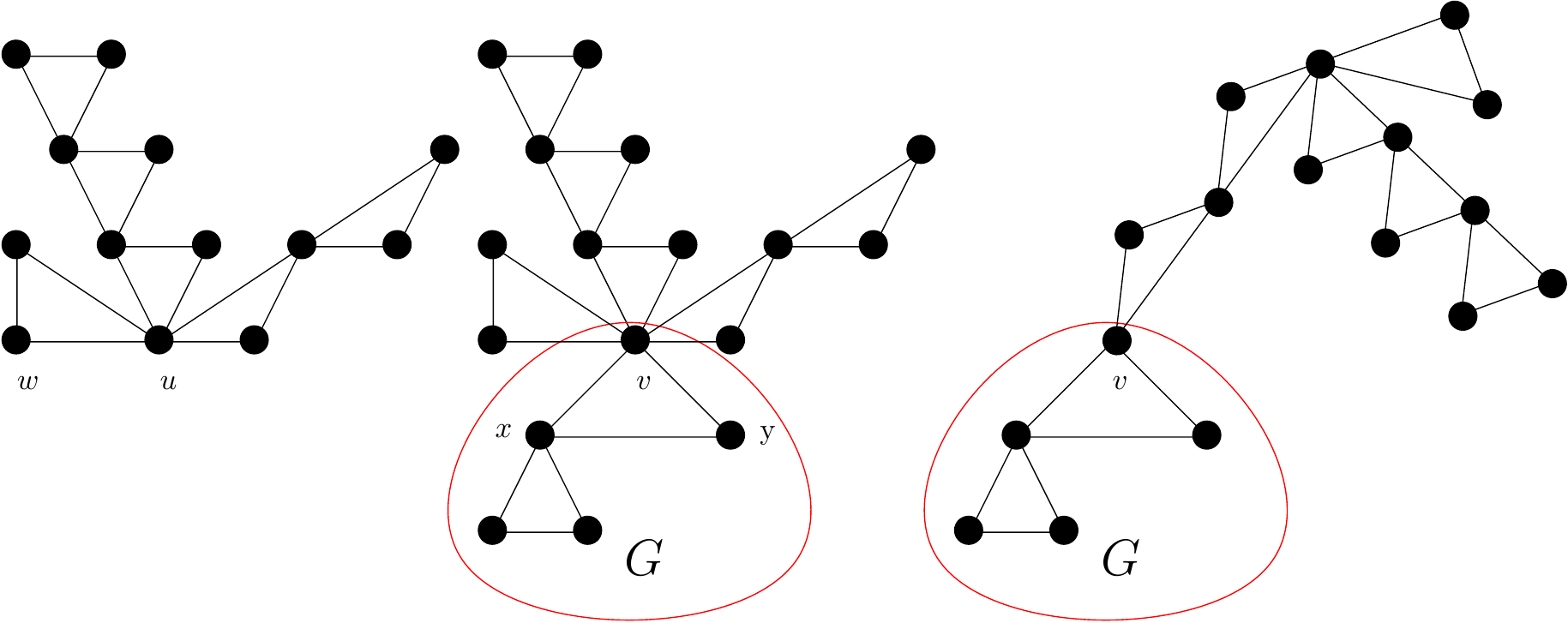}   
	\caption{From left to right $H$, $H_1$ and $H_2$.}
	\label{fig: prop2}
\end{figure}
To ease the reading of the next proposition we recommend to see Fig.~\ref{fig: prop2}.
\begin{proposition}\label{prop: pendant pathas at the same vertex}
	Let $G$ be a connected graph, and let $u\in V(G)$ such that $G-u$ is connected. Let $H$ be the graph obtained from $S(k(q-1)+1,k)$ by adding for all $1\le i \le k$ one pendant $(q,b_i)$-path-block (possible empty, i.e., $b_i=0$) to each leaf block,  let $v\in V(H)$ be the vertex of degree $k(q-1)$, and let $w\in V(H)$ be any vertex in  leaf block of $H$. If $H_1$ is the graph obtained by the coalescence between $G$ and $H$ at $u$ and $v$, $H_2$ is the graph obtained by the coalescence between $G$ and $H$ at $u$ and $w$, then $H_1\prec H_2$ .
\end{proposition}
\begin{proof}
	Observe that $H-w$ is  connected and $H-v$ is a disconnected spanning subgraph of $H-w$. Thus, by Lemma~\ref{lem: subgraph poly} $H-w\prec H-v$. Therefore, the result follows from Lemma~\ref{lem: first technical lemma}.
\end{proof}
The following proposition play a central role to prove the main result of this article. We use $G(r,s)$ to denote the graph obtained by means of the coalescence between $G$ and a copy of $K_r$ at $u\in V(G)$ and any vertex of the complete graph, and between $G$ and $K_s$ at $v\in V(G)$ and any vertex of the complete graph (see the example depicted in Fig.~\ref{fig: pendant cliques}).
\begin{figure}
	\centering
	\includegraphics[scale=0.5]{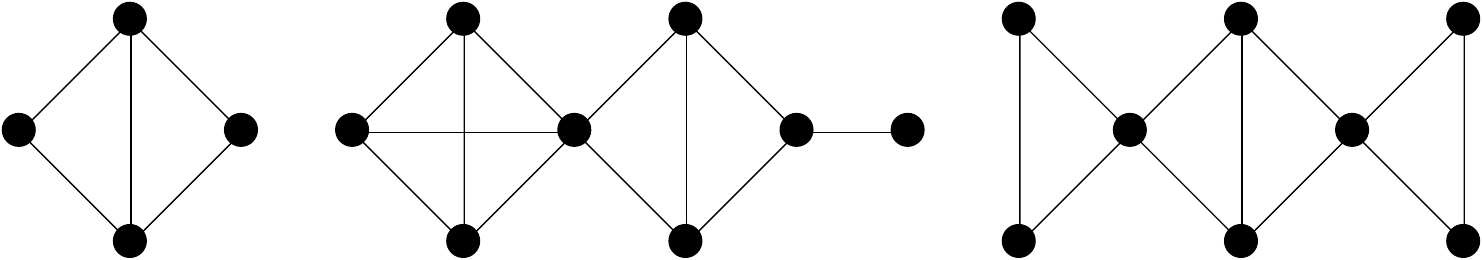}   
	\caption{From left to right $G$, $G(4,1)$ and $G(3,3)$.}
	\label{fig: pendant cliques}
\end{figure}
\begin{proposition}\label{prop: pendant cliques}
	Let $G$ be a connected graph and let $u,v\in V(G)$. If $r$ and $s$ are two integers such that $1 \le r\le s-2$, $G-u\prec G-v$ or $G-u=G-v$, then $G(r,s)\prec G(r+1,s-1)$. 
\end{proposition}
\begin{proof}
	By applying Lemma~\ref{lem: Schwenk} to $G(r,s)$ at $v$ we obtain
	\begin{equation}\label{eq: lema cliques 1}
		P_{G(r,s)}(x)=(x+1)^{s-2}[(x-s+2)P_{G(r,s)-K_{s-1}}(x)-(s-1)P_{G(r,s)-K_s}(x)].
	\end{equation}
	Applying again Lemma~\ref{lem: Schwenk} to $P_{G(r,s)-K_{s-1}}(x)$ and $P_{G(r,s)-K_s}(x)$ we obtain
	\begin{equation*}
	\begin{split}
	P_{G(r,s)}(x)& =(x+1)^{s+r-4}\{(x-s+2)[(x-r+2)P_G(x)-(r-1)P_{G-u}(x)]\\
	&-(s-1)[(x-r+2)P_{G-v}(x)-(r-1)P_{G-\{u,v\}}(x)]\}.
	\end{split}
	\end{equation*}
	By symmetry
	\begin{equation*}
	\begin{split}
	P_{G(r+1,s-1)}(x)& =(x+1)^{s+r-4}\{(x-s+3)[(x-r+1)P_G(x)-rP_{G-u}(x)]\\
	&-(s-2)[(x-r+1)P_{G-v}(x)-rP_{G-\{u,v\}}(x)]\}.
	\end{split}
	\end{equation*}
	Hence
	\begin{equation}\label{eq: proposicion cliques}
	\begin{split}
	P_{G(r+1,s-1)}(x)-P_{G(r,s)}(x)&=(x+1)^{s+r-4}\{(s-r-1)[P_G(x)+P_{G-u}(x)\\
	&+P_{G-v}(x)+P_{G-\{u,v\}}(x)]+(x+1)(P_{G-v}(x)\\
	&-P_{G-u}(x))\}.
	\end{split}
	\end{equation}
Therefore, if  $1\le r\le s-2$, by~\eqref{eq: proposicion cliques} and Lemma~\ref{lem: subgraph spectral radius}, $G(r,s)\prec G(r+1,s-1)$. 
%
\end{proof}

A $(q,b)$-path-block in $\B{n}$ have blocks $B_1,\ldots,B_b$ such that $V(B_i)\cap V(B_{i+1})=\{v_i\}$ for every $1\le i\le b-1$ and $V(B_i)\cap V(B_j)=\emptyset$ whenever $1\le i<j\le n$ and $|i-j|>1$. Let $G$ be a graph and let $v,w\in V(G)$ be two adjacent vertices. We use $G[q,k,\ell]$ to denote the graph obtained by adding a pendant $(q,\ell)$-path-block at $v$ and a pendant $(q,k)$-path-block at $w$,  where $1\le\ell\le k$, and $G[q,r,0]$ stands for the graph obtained from $G$ by adding just a $(q,r)$-path block at $w$. By ``pendant at $v$'' we mean identifying a noncut vertex from one of the two leaf blocks with a noncut vertex $v\in V(B_1)$ (see Fig.~\ref{fig: prop4}). 
\begin{figure}
	\centering
	\includegraphics[scale=0.5]{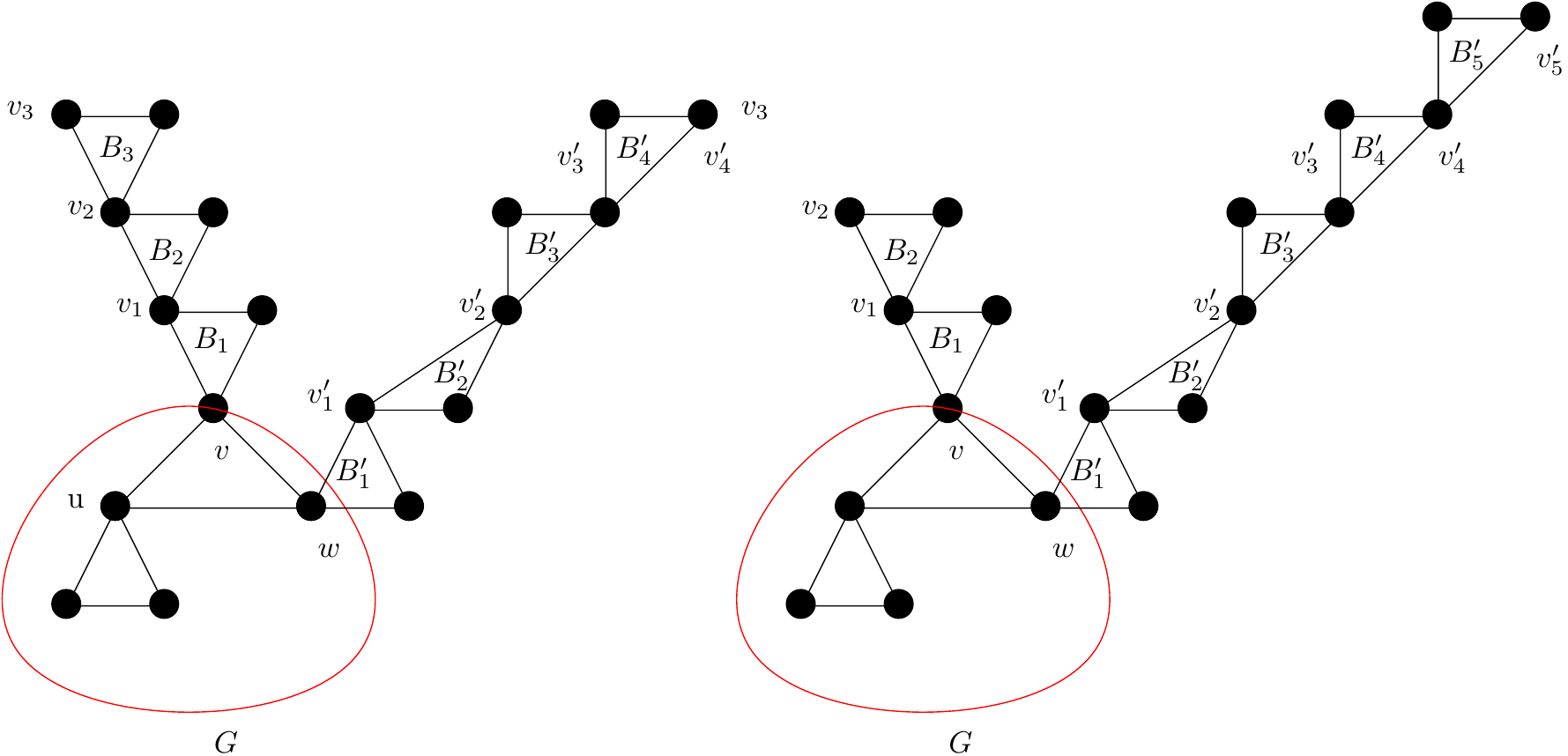}   
	\caption{From left to right $G[3,3,4]$ and $G[3,2,5]$}
		\label{fig: prop4}
\end{figure}

\begin{proposition}\label{prop: pendant path bocks at distinct vertices}
	Let $G\in \B{n}$ with at least one cut vertex. If $\ell$ and $k$ are two positive integers such that $1\le\ell\le k$ and $G[q,k,\ell]$ has at most three adjacent cut vertices, then $G[q,k,\ell]\prec G[q,k+1,\ell-1]$.
\end{proposition}

\begin{proof}
		We use $B_1,\ldots,B_{\ell}$ and $B'_1,\ldots,B'_k$ to denote the blocks of the $(q,\ell)$-path-block, denoted $\camino{\ell}$, and the $(q,k)$-path-block, denoted $\camino{k}$, respectively, $v=v_0$, $w=v'_0$, and $v_{\ell}$ and $v'_k$ stands for a fixed arbitrary noncut vertex of $B_{\ell}$ and $B'_k$ respectively. By $G^l[q,k,\ell]$ (resp. $G^r[q,k,\ell]$) we denote the graph $G[q,k,\ell]-v_{\ell}$ (resp. $G[q,k,\ell]-v'_{k}$). When both vertices are removed we use $G^{lr}[q,k,\ell]$. By applying Lemma~\ref{lem: Schwenk} to $G^l[q,k,\ell-1]$ at $v_{\ell-1}$ is derived the following identity.
		\begin{equation}\label{eq: uno prop principal}
		P_{G[q,k,\ell]}(x)=(x+1)^{q-1}((x-q+1)P_{G[q,k,\ell-1]}(x)-q P_{G^l[q,k,\ell-1]}(x)).
		\end{equation}
		Analogously
		\begin{equation}\label{eq: dos prop principal}
		P_{G[q,k,\ell]}(x)=(x+1)^{q-1}((x-q+1)P_{G[q,k-1,\ell]}(x)- q P_{G^r[q,k-1,\ell]}(x)).
		\end{equation}
		By combining~\eqref{eq: uno prop principal} and~\eqref{eq: dos prop principal} we obtain
		\begin{equation}\label{eq: tres prop principal}
		P_{G[q,k+1,\ell-1]}(x)-P_{G[q,k,\ell]}(x)=q(x+1)^{q-1}(P_{G^l[q,k,\ell-1]}(x)-P_{G^r[q,k,\ell-1]}(x))).
		\end{equation}
		Again, by using properly Lemma~\ref{lem: Schwenk}, we derive the next identity	
		\begin{equation*}
		\begin{split}
		P_{G^l[q,k,\ell]}(x)&=(x+1)^{2q-3}\{(x-q+1)[(x-q+2)P_{G[q,k-1,\ell-1]}(x)\\
		&-(q-1)P_{G^l[q,k-1,\ell-1]}(x)]+q((q-1)P_{G^{lr}[q,k-1,\ell-1]}(x)\\
		&-(x-q+2)P_{G^r[q,k-1,\ell-1]}(x))\}.
		\end{split}
		\end{equation*}
		Analogously,
		\begin{equation*}
		\begin{split}
		P_{G^r[q,k,\ell]}(x)&=(x+1)^{2q-3}\{(x-q+1)[(x-q+2)P_{G[q,k-1,\ell-1]}(x)\\
		&-(q-1)P_{G^r[q,k-1,\ell-1]}(x)]+q((q-1)P_{G^{lr}[q,k-1,\ell-1]}(x)\\
		&-(x-q+2)P_{G^l[q,k-1,\ell-1]}(x))\}.
		\end{split}
		\end{equation*}
		Hence
		\begin{equation}\label{eq: base case}
				\begin{split}
				P_{G^{l}[q,k,\ell]}(x)-P_{G^r[q,k,\ell]}(x)&=(x+1)^{2(q-1)}(P_{G^{\ell}[q,k-1,\ell-1]}(x)\\
				&-P_{G^r[q,k-1,\ell-1]}(x)).
				\end{split}
		\end{equation}
		By applying~\eqref{eq: base case} repeatedly we obtain for every $1\le j\le\ell$
		\begin{equation}\label{eq: general formula}
		\begin{split}
		P_{G^{l}[q,k,\ell]}(x)-P_{G^r[q,k,\ell]}(x)&=(x+1)^{2j(q-1)}(P_{G^{l}[q,k-j,\ell-j]}(x)\\
		&-P_{G^r[q,k-j,\ell-j]}(x)).
		\end{split}
		\end{equation} 
		Replacing in~\eqref{eq: tres prop principal} we obtain that for every $0\le j\le\ell-1$
		\begin{equation}\label{eq: formula principal}
		\begin{split}
		P_{G[q,k+1,\ell-1]}(x)-P_{G[q,k,\ell]}(x)&=q(x+1)^{(2j+1)(q-1)}(P_{G^l[q,k-j,\ell-j-1]}(x)\\
		&-P_{G^r[q,k-j,\ell-j-1]}(x)).
		\end{split}
		\end{equation}
		In particular, if $j=\ell-1$
		\begin{equation}\label{eq: formula final}
		\begin{split}
		P_{G[q,k+1,\ell-1]}(x)-P_{G[q,k,\ell]}(x)&=q(x+1)^{(2\ell-1)(q-1)}(P_{(G-v)[q,k-\ell+1,0]}(x)\\
		&-P_{G^r[q,k-\ell+1,0]}(x)).
		\end{split}
		\end{equation}
Thus it suffices to prove that $G^r[q,t,0]\prec (G-v)[q,t,0]$ for all $t\ge 1$ (see Fig.~\ref{fig: prop4bis}). First observe that, since $G$ has at least one cut vertex, there exists a graph $H_1\in\B{n}$ such that $G^r[q,t,0]$ is the coalescence between $H_1$ and $\camino{t+1}-v_{t+1}$ at a noncut vertex $x\in V(H_1)$ and a noncut vertex $v_0\in V(B_1)$ of $\camino{t+1}-v_{t+1}$ respectively. Analogously, $(G-v)[q,t,0]$ is the coalescence between $H_1$ and  $\camino{t+1}-v_0$ at a noncut vertex $x\in V(H_1)$ and a noncut vertex $y\in V(B_1)\setminus\{v_0\}$ of $\camino{t+1}-v_0$ respectively. From this observation combined with Lemma~\ref{lem: second technical lemma} and Proposition~\ref{prop: pendant cliques} we conclude that $G^r[q,t,0]\prec (G-v)[q,t,0]$. 
\end{proof}
\begin{figure}
	\centering
	\includegraphics[scale=0.5]{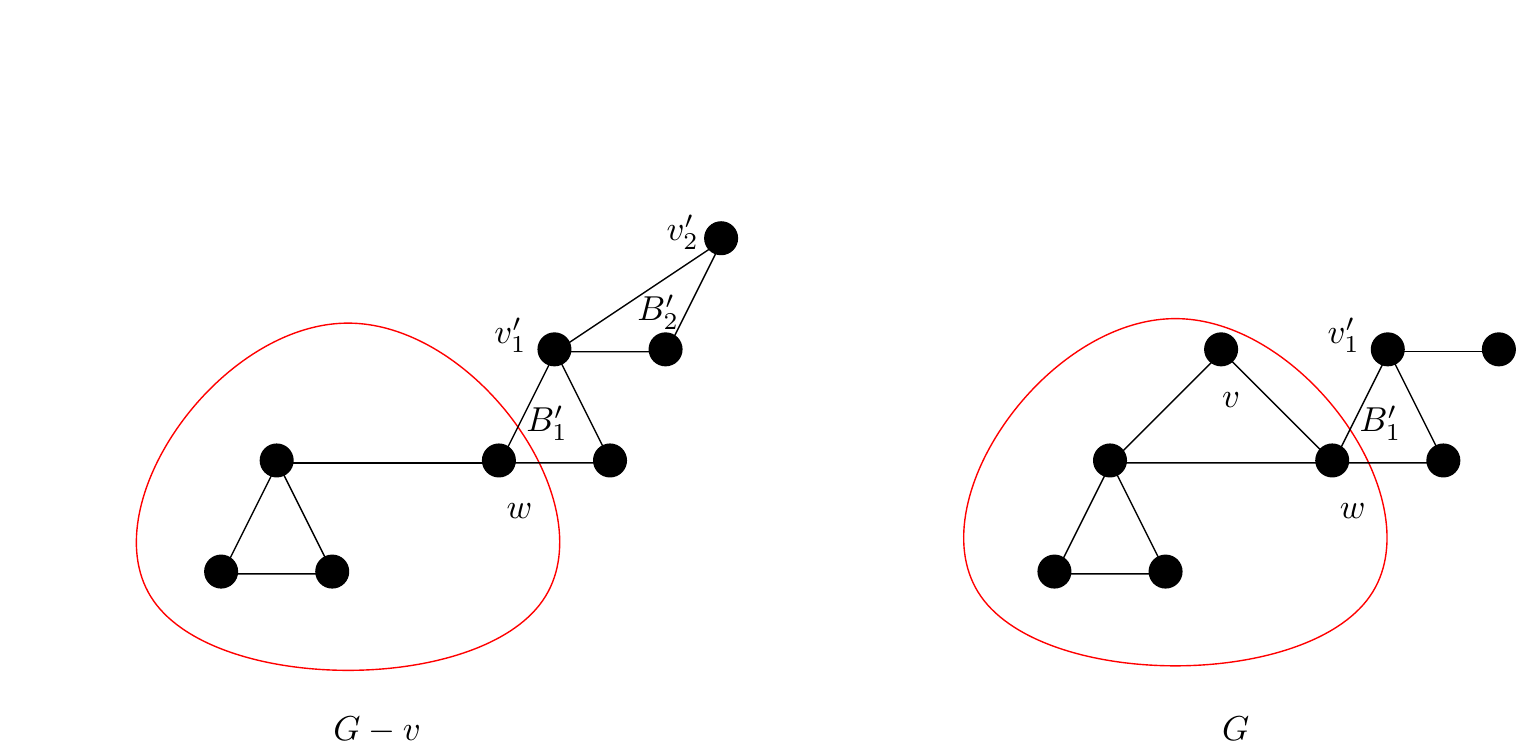}   
	\caption{From left to right $(G-v)[3,2,0]$ and $G^r[3,2,0]$.}
	\label{fig: prop4bis}
\end{figure}
\section{Main result}\label{sec: main result}
Let $G\in\B{n}$ and let $B$ be a block of $G$. We say that $B$ is a \emph{special block of type one} if  $B$ has at least two pendant path-blocks at $v\in V(B)$ (see Fig.~\ref{fig: prop2}, the block whose vertex set is $\{v,x,y\}$ is a special block of type one). We say that $B$ is a \emph{special of type two} if $B$ has a pendant path-block at $v\in V(B)$ and a pendant path-block at $w\in V(B)$ with $v\neq w$ (see Fig.~\ref{fig: prop4}, the block whose vertex set is $\{u,v,w\}$ is a special block of type two). The below lemma, whose proof is omitted, will be used to prove our main result.
\begin{lemma}\label{lem: special blocks}
	If $G\in\B{n}$, then $G$ either is a $(q,b)$-path-block, or has a special block of type one, or has a special block of type two.
\end{lemma}  
%
%
%
Now we are ready to put all pieces together in order to prove the main result of the article.
\begin{proof}[Proof of Theorem~\ref{thm:extremal radius}]
	The upper bound follows from Theorem~\ref{thm:extremal radius maximum}. Assume that $G\in\B{n}$ and it is not a path-block. By Lemma~\ref{lem: special blocks}, $G$ has either a special block of type one or a special block of type two. Hence, by Propositions~\ref{prop: pendant pathas at the same vertex} and~\ref{prop: pendant path bocks at distinct vertices}, there exists a graph transformation onto $G$, involving the corresponding pendant-path-blocks, such that the resulting graph $G'$ satisfies $\rho(G')<\rho(G)$ and $G'$ has an special block less than $G$. Therefore, continuing with this procedure as long as $G'$ is a path-block, we conclude that  $\rho\left(\camino{\frac{n-1}{q}}\right)<\rho(G)$ for all $G\in\B{n}$.
\end{proof}
\section{Lower bound for the spectral radius}\label{sec: bounds}
\begin{theorem}
	Let $G\in\B{n}$ and let $n-1>q$. Then, 
	\[\rho(G) \ge q+\frac{\sqrt{q}}{2}\]
    for every $2\le q \le 4$, and 
    \[\rho(G) \ge q+\frac{4+(q-1)\sqrt{2}}{q+3\sqrt{2}},\]
    for every $q \ge 5$. 
\end{theorem}
\begin{proof}
%
	Assume now that $q\ge 2$ and $b\ge 3$. By  Lemma~\ref{lem: subgraph spectral radius}, we know that $\rho(P_{3}^{q})\le\rho(\camino{b})\le\rho(G)$ for every graph $G\in\B{n}$. By simple calculation, using Lemma~\ref{lem: Schwenk}, we obtain the characteristic polynomial of $\camino{3}$
	\begin{equation}
	\begin{split}
	P_{\camino{3}}(x)&=(x+1)^{3q-4}\Big((x-q)(x+2)+1\Big)\\&\Big((x-q)\Big((x-q)(x+2)+1\Big)-2q\Big).
	\end{split}
	\end{equation}
	Since $(x-q)\Big((x-q)(x+2)+1\Big)-2q = -2q$ when $\Big((x-q)(x+2)+1\Big)=0$, we have that $\rho(P_{3}^{q})$ is the greatest root of $f_q(x) := (x-q)\Big((x-q)(x+2)+1\Big)-2q$. Futhermore, since $f_q(x)$ is an increasing function on $(q, +\infty)$ and $f_q(q)<0$, we have that $\rho(P_{3}^{q})$ is the unique root of $f_q(x)$ on $(q, +\infty)$. 
    
    Using the following two facts
    $$
    f_q\left(q+\frac{\sqrt{q}}{2}\right)= \frac{q+4}{8}\sqrt{q} + \frac{q(q-6)}{4}\le 0\quad \quad\mbox{ and } \quad \quad f_q(q+1) = 4-q\ge 0,
    $$ 
for every $2\le q\le 4$, and 

    $$
    f_q(q+1) = 4-q\le 0\quad \quad\mbox{ and } \quad \quad  f_q(q+\sqrt{2}) = 4+3\sqrt{2}\ge 0,
    $$ 
    for every $q \ge 5$, and taking into account that $f''_q(x)>0$ for every $x\in(q,+\infty)$ we conclude that 
    $$\rho(G) \ge q+\frac{\sqrt{q}}{2}$$
    for every $2\le q \le 4$, and 
    $$\rho(G) \ge q+\frac{4+(q-1)\sqrt{2}}{q+3\sqrt{2}},$$
    for every $q \ge 5$.

\end{proof}

\section{Discussions and further research}\label{sec: discussion}

 We have presented three graphs transformations to deal with the minimum spectral radius of this class of block graphs, namely Propositions~\ref{prop: pendant pathas at the same vertex},~\ref{prop: pendant cliques} and~\ref{prop: pendant path bocks at distinct vertices}, but the last one has very strong hypothesis on the graph $G$. We do not know if they can be weakened. Nevertheless, we have collected very strong computational evidence that drives us to following conjecture.
\begin{conjecture}\label{conj: primera}
	If $G\in\B{n}\setminus\{\camino{b}\}$, then $G\prec \camino{b}$.
\end{conjecture}
Consequently, if this statement were true the following weaker conjecture would be also true.
\begin{conjecture}
	If $G\in\B{n}\setminus\{\camino{b}\}$, then $\rho(\camino{b})< \rho(G)$.
\end{conjecture}
We believe that in for proving Conjecture~\ref{conj: primera} knew graph transformations need to be developed. 

Another interesting graph class, related to that considered in this paper, to study the problem of finding the maximum and minimum spectral radius is that formed by those block graphs on $n$ vertices having exactly $b$ blocks not necessary all of them with the same size.

\section*{Acknowledgments}

Cristian M. Conde acknowledges partial support from ANPCyT PICT 2017-2522. Ezequiel Dratman and Luciano N. Grippo acknowledge partial support from ANPCyT PICT 2017-1315.

\bibliographystyle{abbrv}
\bibliography{blocks_same_size_ARXIV} 
\end{document}